\documentclass[graybox]{svmult}

\usepackage{mathptmx}       % selects Times Roman as basic font
\usepackage{helvet}         % selects Helvetica as sans-serif font
\usepackage{courier}        % selects Courier as typewriter font
\usepackage{type1cm}        % activate if the above 3 fonts are
\usepackage{makeidx}         % allows index generation
\usepackage{graphicx}        % standard LaTeX graphics tool
\usepackage{multicol}        % used for the two-column index
\usepackage[bottom]{footmisc}% places footnotes at page bottom

\usepackage{hyperref}
\usepackage{amsmath}
\usepackage[all]{xy}

\newtheorem{constr}{Construction}

\newcommand{\PP}{\mathbf{P}}

\newcommand{\OO}{\mathcal{O}}

\newcommand{\Ncal}{\mathcal{N}}

\newcommand{\family}{\underline{\Lambda}}

\newcommand{\Hom}{\mathrm{Hom}}
\newcommand{\rank}{\mathrm{rk}}

\newcommand{\Hilb}{\text{\upshape{Hilb}}}

\newcommand{\Spec}{\mathrm{Spec}}

\newcommand{\codim}{\mathrm{codim}}

\newcommand{\Supp}{\mathrm{Supp}}

\newcommand{\rk}{\mathrm{rank}}

\makeindex             % used for the subject index
\begin{document}

\title*{Focal schemes to families of secant spaces to canonical curves}
% Use \titlerunning{Short Title} for an abbreviated version of
% your contribution title if the original one is too long
\author{Michael Hoff}
% Use \authorrunning{Short Title} for an abbreviated version of
% your contribution title if the original one is too long
\institute{Michael Hoff \at Universit\"at des Saarlandes, Campus E2 4, D-66123 Saarbr\"ucken, Germany
\\
\email{hahn@math.uni-sb.de}}
%
% Use the package "url.sty" to avoid
% problems with special characters
% used in your e-mail or web address
%
\maketitle

\abstract*{For a general canonically embedded curve $C$ of genus $g\geq 5$, let $d\le g-1$ be an integer such that the Brill--Noether number $\rho(g,d,1)=g-2(g-d+1)\geq 1$. We study the family of $d$-secant $\PP^{d-2}$'s to $C$ induced by the smooth locus of the Brill--Noether locus $W^1_d(C)$. 
%In \cite{CS00}, Ciliberto and Sernesi showed that one can recover a general canonical curve of odd genus from this family if $\dim(W^1_d(C))=\rho(g,d,1)=1$. 
Using the theory of foci and a structure theorem for the rank one locus of special $1$-generic matrices by Eisenbud and Harris, we prove a Torelli-type theorem for general curves by reconstructing the curve from its Brill--Noether loci $W^1_d(C)$ of dimension at least $1$.}

\abstract{
For a general canonically embedded curve $C$ of genus $g\geq 5$, let $d\le g-1$ be an integer such that the Brill--Noether number $\rho(g,d,1)=g-2(g-d+1)\geq 1$. We study the family of $d$-secant $\PP^{d-2}$'s to $C$ induced by the smooth locus of the Brill--Noether locus $W^1_d(C)$. 
%In \cite{CS00}, Ciliberto and Sernesi showed that one can recover a general canonical curve of odd genus from this family if $\dim(W^1_d(C))=\rho(g,d,1)=1$. 
Using the theory of foci and a structure theorem for the rank one locus of special $1$-generic matrices by Eisenbud and Harris, we prove a Torelli-type theorem for general curves by reconstructing the curve from its Brill--Noether loci $W^1_d(C)$  of dimension at least $1$. 
}


\noindent\subclassname{14H51, 14M12, 14C34}
%14H51: Algebraic geometry, curves, special divisors (gonality, Brill-Noether theory)
%14M12: Algebraic geometry, special varieties, determinantal varieties
%14C34: Algebraic geometry, cycles and subschemes, Torelli problem

\section{Introduction and motivation}

For a general canonically embedded curve $C$ of genus $g\ge 5$ over $\mathbf{C}$, we study the local structure of the Brill--Noether locus $W^1_d(C)$ for an integer $\lceil \frac{g+3}{2} \rceil \le d\le g-1$. 
Our main object of interest is the focal scheme associated to the family of $d$-secant $\PP^{d-2}$'s to $C$. The focal scheme arises in a natural way as the degeneracy locus of a map of locally free sheaves associated to a family of secant spaces to a curve. In other words, the focal scheme (or the scheme of first-order foci) consists of all points where a secant intersects its infinitesimal first-order deformation. 

In \cite{CS92} and \cite{CS95}, Ciliberto and Sernesi studied the geometry of the focal scheme associated 
to the family of $(g-1)$-secant $\PP^{g-3}$'s induced by the singular locus $W^1_{g-1}(C)$ of the theta divisor, and they gave a conceptual new proof of Torelli's theorem.  
Using higher-order focal schemes for general canonical curves of genus $g=2m+1$, 
they showed in \cite{CS00} that the family of $(m+2)$-secants induced by $W^1_{m+2}(C)$ also determines the curve. These are the extremal cases, that is, the degree $d$ is maximal or minimal with respect to the genus $g$ (in symbols $d=g-1$ or $d=\frac{g+3}{2}$ and $g$ odd). 
The article \cite{Baj10} of Bajravani can be seen as a first extension of the previous results to another Brill--Noether locus ($g=8$ and $d=6=\lceil \frac{g+3}{2} \rceil$). 

Combining methods of \cite{CS95},\cite{CS00} and \cite{CS10}, we will give a unified proof which shows that the canonical curve is contained in the focal schemes parametrised by the smooth locus of any $W^1_d(C)$ if $d\leq g-1$ and $\rho(g,d,1)=g-2(g-d+1)\geq 1$. Moreover, we have the following Torelli-type theorem (see also Corollary \ref{corollaryRecovering}). 

\begin{theorem}
A general canonically embedded curve of genus $g$ can be reconstructed from its Brill--Noether locus $W^1_d(C)$ if $\lceil \frac{g+3}{2} \rceil \le d\le g-1$. 
\end{theorem}

In \cite{PTiB}, G. Pirola and M. Teixidor i Bigas proved a generic Torelli-type theorem for $W^r_d(C)$ 
if $\rho(g,d,r)\ge 2$, or $\rho(g,d,r)=1$ and $r=1$. 
Whereas they used the global geometry of the Brill--Noether locus to recover the curve, 
our theorem is based on the local structure around a smooth point of $W^1_d(C)\subset W_d(C)$. 
Only first-order deformations are needed. 

Our proof follows \cite{CS00}. 
We show that the first-order focal map is in general $1$-generic and 
apply a result of D. Eisenbud and J. Harris \cite{EH91} 
in order to describe the rank one locus of the focal matrix. Two cases are possible. The rank one locus of the focal matrix consists either of the support of a divisor $D$ of degree $d$ corresponding to a line bundle $\OO_C(D)\in W^1_d(C)$ or of a rational normal curve. Even if we are not able to decide which case should occur on a general curve (see Section \ref{firstorderFocalMap} for a discussion), we finish our proof by studying focal schemes to a family of rational normal curves induced by the first-order focal map.

In Section \ref{focitheory}, 
we recall the definition of focal schemes as well as general facts and known results about focal schemes. 
Section \ref{FociHigherDimBNLoci} is devoted to prove the generalisation of the main theorem of \cite{CS00} 
to arbitrary positive dimensional Brill--Noether loci. 

\section{The theory of foci}
\label{focitheory}
% Always give a unique label
% and use \ref{<label>} for cross-references
% and \cite{<label>} for bibliographic references
% use \sectionmark{}
% to alter or adjust the section heading in the running head
We recall the definition as well as the construction of the family of $d$-secant $\PP^{d-2}$'s 
induced by an open dense subset of $C^1_d$. 
Afterwards we introduce the characteristic or focal map and 
define the scheme of first-order foci of rank $k$ associated to the above family. 
We give a sligthly generalised definition of the scheme of first- and second-order foci compared to \cite{CS95}. 
In Section \ref{propFocalScheme}, we recall the basic properties of the scheme of first-order foci. 
Our approach follows \cite{CS10}.

%%%%%%%%%%%%%%%%%%%%%%%%%%%%%%%%%%%%%%%%%%%%%%%
\subsection{Definition of the scheme of first-order foci}

 Let $C$ be a Brill--Noether general canonically embedded curve of genus $g\geq 5$, 
 and let $d\le g-1$ be an integer such that the Brill--Noether number $\rho:=\rho(g,d,1)=g-2(g-d+1)\geq 1$. 
 Let $C^1_d$ be the variety parametrising effective divisors of degree $d$ on $C$ moving in a linear system of dimension at least $1$ (see \cite[IV, \S 1]{ACGH}). Let $\Sigma\subset W^1_d(C)$ be the smooth locus of $W^1_d(C)$. Furthermore, let $\alpha_d:C^1_d \to W^1_d(C)$ be the Abel-Jacobi map (see \cite[I, \S 3]{ACGH}) and let $S=\alpha_d^{-1}(\Sigma)$. 
 Then $\alpha: S \to \Sigma$ is a $\PP^1$-bundle, and in particular $S$ is smooth of pure dimension $\rho +1$.
 For every $s\in S$, we denote by $D_s$ the divisor of degree $d$ on $C$ defined by $s$ and 
 $\Lambda_s=\overline{D_s}\subset\PP^{g-1}$ its linear span, which is a $d$-secant $\PP^{d-2}$ to $C$. 
 We get a $(\rho +1)$-dimensional family of $d$-secant $\PP^{d-2}$'s parametrised by $S$: 
 $$
 \begin{xy}
  \xymatrix{
  \family \subset S\times \PP^{g-1} \ar[d]^{p} \ar[r]^{\ \ \ \ q} & \PP^{g-1} \\
  S & 
  }
 \end{xy}
 $$
 We denote by $f:\family \to \PP^{g-1}$ the induced map. 

\begin{constr}[of the family $\family$]
 Let ${\bf D}_d\subset C_d\times C$ be the universal divisor of degree $d$ and let ${\bf D}_S\subset S\times C$ 
 be its restriction to $S\times C$. We denote by $\pi:S\times C \to S$ the projection. 
 We consider the short exact sequence
 $$
  0 \to \OO_{S\times C} \to \OO_{S\times C}({\bf D}_S) \to \OO_{{\bf D}_S}({\bf D}_S) \to 0.
 $$
 By Grauert's Theorem, the higher direct image $R^1\pi_*(\OO_{{\bf D}_S}({\bf D}_S)=0$ vanishes 
 and we get a map of locally free sheaves on $S$
 $$
  R^1\pi_*(\OO_{S\times C})\to R^1\pi_*(\OO_{S\times C}({\bf D}_S)) \to 0
 $$
 whose kernel is a locally free sheaf $\mathcal{F}\subset R^1\pi_*(\OO_{S\times C})\cong \OO_S\otimes H^1(C,\OO_C)$ 
 of rank $d-1 = g - (g-d+1)$. The family $\family$ is the associated projective bundle 
 $$
 \family=\PP(\mathcal{F})\subset S\times \PP^{g-1}.
 $$
\end{constr}

\begin{remark}\label{familyAsRulingOfTangentCones}
We can also construct the family $\family$ from the Brill--Noether locus $W_d(C)$ and its singular locus $W^1_d(C)$. At a singular point $L\in W^1_d(C)\backslash W^2_d(C)$, the projectivised tangent cone to $W_d(C)$ at $L$ in the canonical space $\PP^{g-1}$ coincides with the scroll 
$$
X_L=\bigcup_{D\in |L|} \overline{D}
$$ 
swept out by the pencil $g^1_d=|L|$. Hence, the ruling of $X_L$ is the one-dimensional family of secants induced by $|L|$. Varying the point $L$ yields the family $\family$. See also \cite[Theorem 1.2]{CS95}. We conclude that the family $\family$ is determined by $W_d(C)$ and its singular locus $W^1_d(C)$. 
%and the mathscinet entry of \cite{CS00}. Check \cite{CS92}, relation Abel-Jacobi map and Albanese map} 
\end{remark}

In order to define the first-order focal map of the family $\family$, we make a short digression. We consider a flat family $F$ of closed subschemes of a projective scheme $X$ over a base $B$, that is, 
 $$
 \begin{xy}
  \xymatrix{
  F \subset B\times X \ar[d]^{\pi_1} \ar[r]^{\ \ \ \ \pi_2} & X \\
  B & 
  }
 \end{xy}
 $$
Let $T(\pi_1)|_F := \pi_1^*(T_B)|_F$ be the tangent sheaf along the fibers of $\pi_2$ restricted to the family $F$ and let $\Ncal_{F/B\times X}$ be the normal sheaf of $F\subset B\times X$. There is a map 
$$
\psi: T(\pi_1)|_F \to \Ncal_{F/B\times X}
$$
called the \emph{global characteristic map} of the family $F$ which is defined by the following exact and commutative diagram:
$$
\begin{xy}
 \xymatrix{
  & & 0 \ar[d] & & \\
  & & T(\pi_1)|_{F} \ar[r]^{\psi} \ar[d] & \Ncal_{F/B\times X} \ar@{=}[d] & \\
 0 \ar[r] & T_{F} \ar[r] \ar[d]^{d(\pi_2|_F)} & T_{B\times X}|_{F} \ar[r] \ar[d] & \Ncal_{F/B\times X} \ar[r] & 0 \\ 
  & (\pi_2|_F)^{*}(T_{X}) \ar@{=}[r] & \pi_2^{*} (T_{X})|_F & & 
 }
\end{xy}
$$
For every $b\in B$ the homomorphism $\psi$ induces a homomorphism 
$$
\psi_b:T_{B,b}\otimes \OO_{\pi_1^{-1}(b)} \to \Ncal_{\pi_1^{-1}(b)/X}
$$
called the \emph{(local) characteristic map} of the family $F$ at a point $b$.
Since $F$ is a flat family, we get a classifying morphism 
$$
\varphi:B \to \Hilb_Y
$$
by the universal property of the Hilbert scheme $\Hilb_Y$. 
The linear map induced by the characteristic map
$$
H^0(\psi_b): T_{B,b}\to H^0(\Ncal_{\pi_1^{-1}(b)/X})
$$
is the differential $d\varphi_b$ at the point $b$ (see also \cite[p. 198 f]{Ful84}). 
Assuming that $B$, $Y$ and the family $F$ are smooth, all sheaves in the above diagram are locally free and by diagram-chasing, it follows that 
$$
\ker(d(\pi_2|_F))=\ker(\psi) \ \ \ \text{ and } \ \ \ 
\dim (\pi_2(F))=\dim(F)-\rank(\ker(\psi)).
$$

We come back to the smooth family $\family$ and fix some notation for the rest of the article. 
Let 
$
\Ncal:=\Ncal_{\family/S\times \PP^{g-1}}
$ 
be the normal bundle of $\family$ in $S\times \PP^{g-1}$ and let 
$
T(p)|_{\family}:=p^{*}(T_{S})|_{\family}
$
be the restriction of the tangent bundle along the fibers of $q$ to $\family$. 
Let 
$$
\chi: T(p)|_{\family} \to \Ncal
$$
be the global charecteristic map defined as above. 
For every $s\in S$ the homomorphism $\chi$ induces a homomorphism 
$$
\chi_s:T_{S,s}\otimes \OO_{\Lambda_s} \to \Ncal_{\Lambda_s/\PP^{g-1}}
$$
also called the \emph{characteristic map} or \emph{first-order focal map} of the family $\family$ at a point $s$.

\begin{remark}
  Fix an $s\in S$. We have $\Lambda_s=\PP(U)$, where $U\subset V=H^1(C,\OO_C)$ is a vector subspace 
  of dimension $d-1$. The normal bundle of $\Lambda_s$ in $\PP^{g-1}$ is given by 
  $$
  \Ncal_{\Lambda_s/\PP^{g-1}}= V/U \otimes \OO_{\Lambda_s}(1)
  $$
  and
  $$
  H^0(\Lambda_s,\Ncal_{\Lambda_s/\PP^{g-1}})=\Hom(U,V/U).
  $$
  The characteristic map is of the form 
  $$
  \chi_s:T_{S,s}\otimes \OO_{\Lambda_s}\to V/U \otimes \OO_{\Lambda_s}(1).
  $$
  Hence, it is given by a matrix of linear form on $\Lambda_s$.
\end{remark}

We  define the first- and the second-order foci (of rank $k$) of a family $\family$.  

\begin{definition}
\label{defFociLoci}\hspace{1em}
\begin{enumerate}
 \item [(a)] Let $V(\chi)_k$ be the closed subscheme of $\family$ defined by  
$$
V(\chi)_k=\{p\in\family|\ \rank(\chi(p))\leq k\}. 
$$
Then, $V(\chi)_k$ is the \textit{scheme of first-order foci of rank $k$} and the fiber of $V(\chi)_k$ over a point $s\in S$ 
$$
(V(\chi)_k)_s=V(\chi_s)_k \subset \Lambda_s
$$
is the \textit{scheme of first-order foci of rank $k$ at $s$}. 
%We call $V(\chi)_k=\{p\in\family|\ \rank(\chi(p))\leq k\}$ 
%the \emph{scheme of first-order foci of rank k}. 
%Furthermore, we call $\chi$ (resp. $\chi_s$) the \emph{focal map (at $s$)}.
\item [(b)] Assume that $V(\chi)_k$ induces a family of rational normal curves $\underline{\Gamma}$, that is, for a general $s\in S$ the fiber $\Gamma_s=V(\chi_s)_k$ is a rational normal curve. 
Let $\psi$ be the global characteristic map of $\underline{\Gamma}$. We call the first-order foci of rank $k$ of the family $\underline{\Gamma}$, that is, 
$$
V(\psi)_k=\{p\in\underline{\Gamma}|\ \rank(\psi(p))\leq k\}
$$ 
the \emph{second-order foci of rank $k$} of the family $\family$.
\end{enumerate}
\end{definition}

\begin{remark}
Our definition of scheme of first-order foci is a slight generalisation of the definition given in \cite{CS92}, \cite{CS95}. Note that if 
$$k=\min\{\rank(T(p)|_{\family}),\rank(\Ncal)\}-1=\min\{\codim S, \codim_{S\times \PP^{g-1}}(\family)\}-1,
$$ 
we get the classical definition of first-order foci. Furthermore, our definition of the second-order foci of rank $k$ is inspired by the definition of higher-order foci of \cite{CS10}.  
\end{remark}

\begin{remark} 
\begin{enumerate}
 \item [(a)] The equality $V(\chi)_s=V(\chi_s)$ is shown in \cite[Proposition 14]{CC93}.
 \item [(b)] If $\chi$ has maximal rank, that is, if $\chi$ is either injective or has torsion cokernel, then 
 $V(\chi)_k$ is a proper closed subscheme of $\family$ for $k\le \min\{\rank(T(p)|_{\family}),\rank(\Ncal)\}-1$. 
 \item [(c)] In Section \ref{FociHigherDimBNLoci} we study the scheme of first-order foci of rank $1$ of the family $\family$. 
\end{enumerate}

\end{remark}

\subsection{Properties of the scheme of first-order foci}
\label{propFocalScheme}

We assume in this section that $C$ is a Brill--Noether general curve. The following proposition is proven in \cite{CS95} which can be easily generalised to the case of divisors of degree $d < g-1$. 

\begin{proposition}
\label{curveInFocalScheme}
For $s\in S$, we have 
$$
D_s\subset V(\chi_s)_1.
$$
In particular, the canonical curve $C$ is contained in the scheme of first-order foci.
\end{proposition}

\begin{proof}
Let $p\in \Supp(D_s)$. Then there exists a codimension $1$ family of effective divisors and hence $d$-secants 
containing the point $p$. Therefore, there is a codimension $1$ subspace $T\subset T_{S,s}$ such that 
the map $\chi_s(p)|_T$ is zero. 
We conclude that the focal map $\chi_s$ has rank at most $1$ in points of $\Supp(D_s)$. \hspace{5.9cm} \qed
\end{proof}

%\begin{remark}
% The proof shows that the canonical curve also lies in the rank $1$ locus of $\chi$. 
% In a general fiber $s\in S$, the points of $\Supp(D_s)$ lie in 
%$D_1(\chi_s)=V(I_{2\times 2}(\chi_s))$. 
%\end{remark}

An important step in the proof of our main theorem is to show that the first-order focal map $\chi_s$ is 
$1$-generic. The general definition of $1$-genericity can be found in \cite{Eis88}. 
In our case, a reformulation of the definition is the following. 

\begin{proposition}
\label{reformulation1generic}
 The matrix $\chi_s$ is $1$-generic if and only if for each nonzero element $v\in T_{S,s}$, 
 the homomorphism 
 $$
 H^0(\chi_s)(v) \in \Hom(U,V/U)
 $$
 is surjective.
\end{proposition}

We recall what is known about the $1$-genericity of the matrix $\chi_s$.

\begin{proposition}[{\cite[Theorem 2.5]{CS95}, \cite[Theorem 2]{CS00}, \cite{Baj10}}]
 Let $s\in S$ be a general point.
 \begin{enumerate}
  \item [(a)] If $D_s$ is a divisor of degree $g-1$ cut on $C$ by $\Lambda_s$, 
  then the matrix $\chi_s$ is $1$-generic (equivalently, $V(\chi_s)_1$ is a rational normal curve) 
  if and only if the pencil $|D_s|$ is base point free. 
  \item [(b)] If $\rho=\rho(g,d,1)=1$, then the matrix $\chi_s$ is $1$-generic 
  (equivalently, $V(\chi_s)_1$ is a rational normal curve). 
  \item [(c)] If $g=8$ and $d=6$, then the matrix $\chi_s$ is $1$-generic.
  \end{enumerate}
\end{proposition}

\begin{remark}[{\cite[p. 253]{Ser06}}]
 Another fact related to the $1$-genericity of $\chi_s$ is the following: 
 Let 
 $$
 \begin{xy}
  \xymatrix{
  \family_{\varepsilon} \subset \Spec(\mathbf{C}[\varepsilon])\times \PP^{g-1} \ar[d] \\
   \Spec(\mathbf{C}[\varepsilon])
  }
 \end{xy}
 $$
 be the first order deformation of $\Lambda_s$ defined by $H^0(\chi_s)(v)$ for a vector $v\in T_{S,s}$. Then, $H^0(\chi_s)(v)$ is surjective if and only if $q(\family_{\varepsilon})\subset \PP^{g-1}$ is not contained in a hyperplane. Furthermore, the definition of the first-order foci at a point $s\in S$ depends only on the geometry of the family  $\family$ in a neighbourhood of $s$. A point in $V(\chi_s)_{k}$ is a point where the fiber $\Lambda_s$ intersects a codimension $k$ family of its infinitesimally near ones.
\end{remark}

\section{Proof of the main theorem}
\label{FociHigherDimBNLoci}
% Always give a unique label
% and use \ref{<label>} for cross-references
% and \cite{<label>} for bibliographic references
% use \sectionmark{}
% to alter or adjust the section heading in the running head
The strategy of the proof is the same as in \cite{CS00}. 
We assume that the canonically embedded curve $C$ is a Brill--Noether general curve. Recall that $g$ and $d$ are chosen such that the Brill--Noether number $\rho:= \rho(g,d,1) \ge 1$. 
We begin by showing some standard properties of a line bundle over a Brill--Noether general curve 
which we will use later on. 
Then we prove that the matrix $\chi_s$ is 1-generic for general $s\in S$ 
and study the rank one locus of $\chi_s$ which will be the divisor $D_s$ or a rational normal curve. 
In the second case, we study the second-order focal locus. 
In both cases we can recover the canonical curve. 

\begin{lemma}
\label{genAssumptionOnC}
Let $C$ be a Brill--Noether general curve and let $L\in W^1_d(C)$ be a smooth point. 
Then $|L|$ is base point free, $H^1(C,L^2)=0$ and 
$g^{\rho+2}_{2d}=|L^2|$ maps $C$ birational to its image 
(it is not composed with an involution). 
\end{lemma}

\begin{proof}
All of our claims follow directly from the generality assumption. 
We just mention that the map induced by $|L^2|$ can not be composed with an irrational involution.  
Hence, if the map is not birational, it is composed with a $g^1_{d'}$ for $d'\leq \frac{2d}{\rho + 2}$ 
%$d'\leq \frac{2d}{\rho + 2} \Rightarrow \rho(g,d',1)\leq \frac{4d}{\rho +2} - g -2 \leq 2d - g - 2 < 0$ since $d\leq g-1$. 
which is impossible for a Brill--Noether general curve. \hspace{3.3cm} \qed
\end{proof}

\begin{corollary}
\label{vanishingOfH1}
Let $C$ be a Brill--Noether general curve and let $L\in W^1_d(C)$ be a smooth point. 
For $i\ge 1$ and $p_1,\dots, p_i \in \Supp(D)$ for $D\in |L|$ general, we have 
$$
h^0(C,L^2(-p_1-\dots-p_i))=2d-i+1-g.
$$ 
In particular, $H^0(C,L^2(-p_1-\dots -p_{\rho+1}))=H^0(C,L)$ and 
$H^1(C,L^2(-p_1-\dots-p_i))=0$ for $i=1,\dots, \rho+1$.
\end{corollary}

\begin{proof} 
$H^0(C,L^2(-p_1-\dots-p_i)) = H^0(C,L^2(-p_1-\dots-p_{i+1}))$ 
if the images under $|L^2|$ of the two points $p_i$ and $p_{i+1}$ are the same point. 
Since $|L^2|$ maps $C$ birational to its image, 
this does not happen for a general choice. \hspace{3.3cm} \qed
\end{proof}

Using Lemma \ref{genAssumptionOnC} and Corollary \ref{vanishingOfH1}, our proof of the following lemma is identical to \cite[Theorem 2]{CS00}. We clarify and generalise the arguments given in \cite[Theorem 2]{CS00}. 

\begin{lemma}
\label{1genericity}
With the assumptions of Lemma \ref{genAssumptionOnC}, the focal matrix 
$
\chi_s:T_{S,s}\otimes \OO_{\Lambda_s} \to \Ncal_{\Lambda_s/\PP^{g-1}}
$
is $1$-generic for a sufficiently general $s\in S$. 
\end{lemma}

\begin{proof} 
 By Proposition \ref{reformulation1generic}, the matrix $\chi_s$ is $1$-generic if and only if 
 for each nonzero element $v\in T_{S,s}$, the homomorphism 
 $
 H^0(\chi_s)(v) \in \Hom(U,V/U)
 $
 is surjective. 
 
 We consider the first order deformation $\family_\varepsilon\subset \Spec(k[\varepsilon])\times \PP^{g-1}$ 
 defined by $H^0(\chi_s)(\theta)$ for a nonzero vector $\theta\in T_{S,s}$. 
 Note that $H^0(\chi_s)(\theta)$ is surjective if and only if the image $q(\family_\varepsilon)\subset \PP^{g-1}$ 
 is not contained in a hyperplane. 
 Let $D_\varepsilon\subset \Spec(k[\varepsilon])\times \PP^{g-1}$ be the first order deformation 
 of the divisor $D_s$ defined by $\theta\in T_{S,s}$. Then 
 $$
 q(\family_\varepsilon)\supset q(D_\varepsilon)
 $$
 and the curvilinear scheme $q(D_\varepsilon)$ corresponding to a divisor on $C$ satisfies 
 $$
 D_s \leq q(D_\varepsilon) \leq 2D_s. 
 $$
 We show for all possible cases that $q(D_\varepsilon)$ is not contained in a hyperplane. 
  \medskip \\
 \underline{Case 1:} The vector $\theta$ is tangent to $\alpha_d^{-1}(L)$, equivalently 
 the family $D_\varepsilon$ deforms the divisor $D_s$ in the linear pencil $|L|$. 
 Let $\varphi_L$ be the morphism defined by the pencil. Then we get 
 $$
 q(D_\varepsilon)=\varphi_L^{*}(\theta),
 $$
 where we identify $\theta$ with a curvilinear scheme of $\PP^1$ supported at the point $s\in \PP^1$. 
 Since $|L|$ is base point free, we have $q(D_\varepsilon)=2D_s$. 
 Therefore,  the curvilinear scheme $q(D_\varepsilon)$ is not contained in a hyperplane 
 since $H^0(C,K_C-2D_s)^* = H^1(C,2D_s)= H^1(C,L^2) = 0$. We are done in this case. 
   \medskip \\
 \underline{Case 2:} We assume that $\theta\in T_{S,s}\backslash\{0\}$ is not tangent to $\alpha_d^{-1}(L)$ at $s$. 
 Let 
 $$
 q(D_\varepsilon)=p_1 + \dots + p_k + 2(p_{k+1} + \dots + p_d)
 $$
 where $D_s=p_1+ \dots + p_d$ and $k\geq 0$. 
    \medskip \\
 \underline{Case 2 (a):} We assume $k\leq \rho$. We have 
 \begin{eqnarray*}
 H^0(C,K_C-q(D_\varepsilon))^* & = & H^1(C,p_1 + \dots + p_k + 2(p_{k+1}+ \dots + p_d)) \\
                               & = & H^1(C,2D_s-p_1- \dots - p_k) \\
                               & = & H^1(C,L^2(-p_1-\dots -p_k)) = 0
 \end{eqnarray*}
 by Corollary \ref{vanishingOfH1}. 
 Hence, the curvilinear scheme $q(D_\varepsilon)$ is not contained in a hyperplane and 
 $H^0(\chi_s)(\theta)$ is surjective. 
    \medskip \\
 \underline{Case 2 (b):} We assume $k\geq \rho + 1$. 
 In the following, we will show that this case can not occur. 
 The vector $\theta$ is also tangent to $p_1 + \dots + p_k + C_{d-k}$. 
 We denote by $E_s$ the divisor $E_s = p_{k+1} + \dots + p_d$. Then
 the tangent space to $p_1 + \dots + p_k + C_{d-k}$ is given by 
 $H^0(E_s,\OO_{E_s}(D_s))$ which is a subspace of $H^0(D_s, \OO_{D_s}(D_s))$.
 The short exact sequence 
 $$
 0 \to \OO_C \to L \to \OO_{D_s}(D_s) \to 0
 $$
 induces a linear map 
 $$
 H^0(D_s,\OO_{D_s}(D_s)) \stackrel{\delta}{\longrightarrow} H^1(C,\OO_C)
 $$
 which we identify with the differential of $\alpha_d$ at $s$ (see \cite[IV, \S 2, Lemma 2.3]{ACGH}). 
 The image of $\theta \in H^0(E_s,\OO_{E_s}(D_s))$ is therefore contained in the linear span of $E_s$. 
 After projectivising, we get 
 $$
 [\delta(\theta)]\in \overline{E_s}= \overline{p_{k+1} + \dots + p_d} \subset \Lambda_s \subset \PP^{g-1}.
 $$
 Since $\theta$ is not tangent to $\alpha_d^{-1}(L)$, the vector $\theta$ is also tangent to $W^1_d(C)$ 
 and therefore the image point $[\delta(\theta)]$ 
 is contained in the vertex $V=T_L(W^1_d(C))$ of $X_L$, the scroll swept out by the linear pencil $|L|$. 
 Hence, for every sufficiently general $D\in |L|$, there is an effective divisor $E$ of degree $d - \rho - 1$ such that 
 $D=E + p_1 + \dots + p_{\rho + 1}$ and $V\cap \overline{E}\neq \emptyset$. 
 Hence, $\dim(\overline{D_s+E})\leq d - 2 + d - \rho - 1$ and equivalently, 
 $$
 h^0(C,D_s+E) = \deg(D_s+E) + \dim(\overline{D_s+E}) + 1 \geq 3.
 $$
 But by Corollary \ref{vanishingOfH1}, $H^0(C,L^2(-p_1-\dots -p_{\rho+1}))=H^0(C,L)$, a contradiction. \hspace{0.1cm} \qed
\end{proof}

Note that 
$$
\chi_s:T_{S,s}\otimes \OO_{\Lambda_s} \to \Ncal_{\Lambda_s/\PP^{g-1}}
$$
is a map between rank $\rho +1$ and $n=h^1(C,L)$ vector bundles 
of linear forms in $\PP^{d-2}=\Lambda_s$. 
Since $d=\rho + 1 + n$ and $\chi_s$ is $1$-generic by Lemma \ref{1genericity}, 
we may apply the following theorem due to Eisenbud and Harris.

\begin{theorem}[{\cite[Proposition 5.1]{EH91}}]
\label{EH92Prop5.1}
 Let $M$ be an $(a+1)\times (b+1)$ $1$-generic matrix of linear forms on $\PP^{a+b}$. 
 If $D_1(M)=\{x\in \PP^{a+b}\ |\ \rank(M(x))\leq 1\}$ contains a finite scheme $\Gamma$ of 
 length $\geq a+b+3$, then $D_1(M)$ is the unique rational normal curve through $\Gamma$ 
 and $M$ is equivalent to the catalecticant matrix.
\end{theorem}
%% They assume that any curve is defined over an algebrically closed field. %%

We get the following corollary. 

\begin{corollary}
\label{rankOneLocusFocalMap}
 For $s\in S$ sufficiently general, the rank one locus $V(\chi_s)_1$ is either $D_s$ or 
 a rational normal curve through $D_s$. 
\end{corollary}

\begin{proof}
 By Lemma \ref{1genericity}, we may apply Theorem \ref{EH92Prop5.1}. 
 Note that $D_s\subset V(\chi_s)_1$ (there exists a codimension $1$ family in $S$ of $\Lambda_s$ 
 containing a point of the support of $D_s$). \hspace{0.5cm} \qed
\end{proof}

\begin{remark}
\begin{enumerate}
 \item [(a)] The scheme of first-order foci at $s\in S$ of the family $\family$ is a secant variety to $V(\chi_s)_1$. 
 \item [(b)] If $d=g-1$ or $\rho=1$, the focal matrix $\chi_s$ is a $2\times (g-3)$ or $n\times 2$-matrix, respectively. 
             Hence, the rank one locus is the scheme of first-order foci, 
             which is a rational normal curve in $\Lambda_s$. 
             We recover the cases of \cite{CS95} and \cite{CS00}.
\end{enumerate}
\end{remark}

\begin{corollary}
Let $C$ be a Brill--Noether general canonically embedded curve. 
If $V(\chi_s)_1=D_s$ for sufficiently general $s\in S$, the family $\family$ determines the canonical curve $C$.  
\end{corollary}

{\bf For the rest of this section, we assume that $\Gamma_s=V(\chi_s)_1$ is a rational normal curve for 
$s\in S$ sufficiently general.}

Let $\Sigma$ be the smooth locus of $W^1_d(C)$ and $L\in \Sigma$.
Let $U\subset \alpha^{-1}_d(L))$ be a Zariski open dense set such that 
$\Gamma_s=V(\chi_s)_1$ for all $s\in U$. We define the surface
$$
\Gamma_L=\overline{\bigcup_{s\in U} \Gamma_s}
$$
and 
$$
\Gamma_{\PP^{g-1}}=\overline{\bigcup_{L\in \Sigma} F_L}.
$$
\\
Let
$$
 \begin{xy}
  \xymatrix{
  \underline{\Gamma} \subset S'\times \PP^{g-1} \ar[d]^{p} \ar[r]^{\ \ \ \ q} & \PP^{g-1} \\
  S' & 
  }
 \end{xy}
$$
be the family induced by all rational normal curves, that is, 
for $s\in S'\subset S$, $\Gamma_s=V(\chi_s)_1$ is a rational normal curve. 
The family $\underline{\Gamma}$ is the rank one locus of the global characteristic map $\chi$ and 
the variety $\Gamma_{\PP^{g-1}}$ is the image of the family $\underline{\Gamma}$ under the second projection $q$.

\begin{remark}
In the cases $d=g-1$ or $\rho=1$ the rational surface $\Gamma_L$ is birational to $\PP^1\times \PP^1\subset \PP^3$ or a quadric cone in $\PP^3$, resprectively. This can be explained in terms of the curve $C$ and the line bundle $L$: 
 \vspace{-0.17cm}
 
For $d=g-1$ we consider the birational image $C'$ of $C\stackrel{|L|\times |\omega_C\otimes L^{-1}|}{\longrightarrow}\PP^1\times \PP^1$ given by the line bundle $L$ and its Serre dual $\omega_C\otimes L^{-1}$. Then the rational surface $\Gamma_L$ is the image of the blow up of $\PP^1\times \PP^1$ along the singular points of $C'$ under the adjoint morphism.  

For $\rho=1$ we consider the birational image $C'$ of the curve $C$ in the quadric cone $Q$ in $\PP^3$ induced by the line bundle $L^2$. Note that $H^0(C,L^2)$ is four-dimensional and the multiplication map $H^0(C,L)\otimes H^0(C,L)\to H^0(C,L^2)$ has a one-dimensional kernel. Then the rational surface $\Gamma_L$ is again the image of the blow up of $Q$ along the singular points of $C'$ under the adjoint morphism. 

We have not found a similar geometrical meaning of the surface $\Gamma_L$ in the other cases 
(see also Question \ref{rank1locusGeneralCurve}). 
\end{remark}

\begin{lemma}
\label{dimensionF}
 The variety $\Gamma_{\PP^{g-1}}$ has dimension at least $3$.
\end{lemma}

\begin{proof}
Note that there is a map $\Gamma_L\to \PP^1=\alpha_d^{-1}(L)$ such that 
the general fiber is a rational curve. Hence, the surface $\Gamma_L$ is rational. 
Assume that $\Gamma_L=\Gamma_{L'}$ for all $L'\in \Sigma$. 
Since the scrolls $X_{L'}$ are algebraically equivalent to each other, the rulings on them cut out 
a $(\rho +1)$-dimensional family of algebraically equivalent rational curves on $\Gamma_{L}$, 
the focal curves. 
(We can also argue that all $d$-secant to $C$ are algebraically equivalent, 
thus the intersection with $\Gamma_L$ yields a $(\rho +1)$-dimensional family of 
algebraically equivalent focal curves.) 
On the desingularization of $\Gamma_L$, all of them are linear equivalent 
since $\Gamma_L$ is regular ($H^1(\Gamma_L,\OO_{\Gamma_L})=0$).   
This implies that all $g^1_d$'s on $C$ are linear equivalent, hence $C$ has a $g^{\rho +1}_d$. 
A contradiction to the generality assumption on $C$. \hspace{9.6cm} \qed
%\Michael{Is the desingularization of $\Gamma_L$ a (blow up of a) Hirzebruch surface? 
%We need something like \cite[Proposition 3]{CS10})}
\end{proof}

For the convenience of the reader, we recall the definition of the second-order foci of the family $\family$ (see also Definition \ref{defFociLoci}). We apply the theory of foci to the family $\underline{\Gamma}\subset S'\times \PP^{g-1}$ and get the characteristic map 
$$
\psi: T(p)|_{\underline{\Gamma}} \to \Ncal_{\underline{\Gamma}/S'\times \PP^{g-1}}
$$
of vector bundles of rank $\rho +1$ and $g-2$, respectively. 
For $s\in S'$, we call the closed subscheme of $\Gamma_s$ defined by 
$\rk(\psi_s)\leq k$ 
the \emph{scheme of second-order foci of rank $k$ at $s$} (of the family $\family$). 

We will show that the scheme of second-order foci of rank $1$ at $s\in S'$ of the family $\family$
is a finite scheme containing the divisor $D_s$ and compute its degree. 

\begin{lemma}
\label{connectionDimRank}
Let
$
\psi_s: T_{S',s}\otimes \OO_{\Gamma_s} \to \Ncal_{\Gamma_s/\PP^{g-1}}
$
be the characteristic map for general $s\in S'$. 
Then the rank of $\psi_s$ at a general point of $\Gamma_s$ is at least 2. %, that is, $\rk_{p\in \Gamma_s}(\psi_s(p))\geq 2$ if $\dim(\Gamma_{\PP^{g-1}})\geq 3$.
\end{lemma}

\begin{proof}
We recall the connection of the rank and the dimension of $\Gamma_{\PP^{g-1}}$ as in \cite[page 6]{CS10}. 
Since $\dim(\Gamma_{\PP^{g-1}}) = \rho + 2 - \rk(\ker(\psi))$, the rank of $\psi_s$ at the general point $p\in \Gamma_s$ is 
\begin{eqnarray*}
\rk(\psi_s(p)) & = & \dim(T(p)|_{\underline{\Gamma}}) - \rk(\ker(\psi)) \\
 & = & \rho + 1 - \rk\ker(\psi)) \\
 & = & \dim(\Gamma_{\PP^{g-1}}) - 1.
\end{eqnarray*}
%\Michael{The $-1$ is missing in \cite[page 9]{CS10}.} 
The lemma follows from Lemma \ref{dimensionF}. 
This fact is also shown in \cite[page 98]{CC93}. \hspace{0.5cm} \qed
\end{proof}

We now consider for a general $s\in S'$ the rank one locus of $\psi_s$ 
which is a proper subset of $\Gamma_s$ by Proposition \ref{connectionDimRank}.

\begin{lemma}
 The degree of $\ V(\psi_s)_1\subset \Gamma_s=V(\chi_s)_1$ is at most $d + \rho$.
\end{lemma}

\begin{proof}
 We imitate the proof of \cite[Theorem 3]{CS00}. 
 Let $s\in S'\subset$ be a general point and let $\Gamma_s\subset \PP^{d-2}=\Lambda_s$ 
 be the rank $1$ locus of the map 
 $$
 \chi_s: T_{S',s} \otimes \OO_{\Lambda_s} \to \Ncal_{\Lambda_s/\PP^{g-1}}.
 $$
 Note that the normal bundle of $\Gamma_s$ splits 
 $$
 \Ncal_{\Gamma_s/\PP^{g-1}} = (\Ncal_{\Lambda_s/\PP^{g-1}}\otimes \OO_{\Gamma_s}) 
 \oplus \Ncal_{\Gamma_s/\Lambda_s} 
 = \OO_{\Gamma_s}(d-2)^{\oplus n} \oplus \OO_{\Gamma_s}(d)^{\oplus d-3}. 
 $$
 Hence, the map $\psi_s$ is given by a matrix 
 $$
 \psi_s =
 \begin{pmatrix}
	A \\ B
 \end{pmatrix}
 $$
 where $A$ is a $n\times (\rho +1)$-matrix and $B$ is a $(d-3) \times (\rho +1)$-matrix. 
 The matrix $A$ represents the map 
 $(\chi_s: T_{S,s}\otimes \OO_{\Lambda_s} \to \Ncal_{\Gamma_s/\PP^{g-1}})|_{\Gamma_s}$ 
 and therefore has rank $1$ and is equivalent to a catalecticant matrix. 
 Let $\{s,t\}$ be a basis of $H^0(\Gamma_s,\OO_{\Gamma_s}(1))$. In an appropriate basis, 
 the matrix $A$ is of the following form 
 \begin{eqnarray*}
  A  =  \begin{pmatrix}
	t^{d-2} & t^{d-3} s & \cdots & t^{d-2-\rho} s^{\rho} \\ 
	t^{d-3}s & \ddots & & t^{d-2-\rho -1} s^{\rho +1} \\ 
	\vdots &  & &  \vdots \\
	t^{d-2-n+1}s^{n-1} & t^{d-2-n} s^{n} & \cdots & s^{d-2}
 \end{pmatrix}  
   = \begin{matrix}  t^{n-1}\cdot \\ t^{n-1}s \cdot \\ \vdots \\ s^{n-1}\cdot \end{matrix} 
 \begin{pmatrix}
	t^{\rho} & t^{\rho-1} s & \cdots & s^{\rho} \\ 
	t^{\rho} & \ddots & & s^{\rho} \\ 
	\vdots &  & &  \vdots \\
	t^{\rho} & t^{\rho-1} s & \cdots & s^{\rho}
 \end{pmatrix}
 .
 \end{eqnarray*}
 We see that the rank $1$ locus of $\psi_s$ is the rank $1$ locus of the following matrix 
 \begin{eqnarray*}
 N=
 \begin{pmatrix}
	t^{\rho} & t^{\rho -1} s & \cdots & s^{\rho}\\
	 &  B
 \end{pmatrix}.
 \end{eqnarray*}
 Since $V(\psi_s)_1\neq \Gamma_s$ by Lemma  \ref{connectionDimRank}, 
 we have 
 $$
 \deg(V(\psi_s)_1) = \deg(D_1(N)) \leq 
 \mathrm{min}\{\mathrm{ degree\ of\ elements\ of\ } I_{2\times 2}(N)\} \le \rho + d.
 $$
 \hspace{11cm} \qed
\end{proof}

%\Michael{I believe that the $n-1$ remaining points should be the intersection of the rational normal 
%curve with the vertex of the scroll. Sernesi mentioned that 
%they proved something similar in \cite{CS95}.}

\begin{proposition}
\label{RNCintersectionVertex}
 Let $s\in L$ be a sufficiently general point. 
 Then, $V(\psi_s)_1$ is the union of $D_s$ and $\rho$ points which are the intersection of 
 $\Gamma_s=V(\chi_s)_1$, and the vertex $V$ of the scroll $X_L$ swept out by the pencil $|L|$. 
\end{proposition}

\begin{proof}
As in the proof of Proposition \ref{curveInFocalScheme}, one can show that
the points in the support of $D_s$ are contained in $V(\psi_s)_1$.

Next, we show that the vertex in $\Lambda_s$ is given by a column of the matrix $\chi_s$. 
Again, we imitate the proof of \cite[Proposition 4.2]{CS95}. 
Each column of the $n\times (\rho +1)$-matrix $\chi_s$ is a section of the rank $n$ vector bundle 
$V/U\otimes \OO_{\Lambda_s}(1)$ (where $U\subset V $ is the affine subspace representing $\Lambda_s$)
corresponding to an infinitesimal deformation of $\Lambda_s$.  
Each section vanishes in a $\rho-1 = (d-2-n)$-subspace of $\Lambda_s$ 
which is a $\rho$-secant of $\Gamma_s$. 
Since $\chi_s$ is $1$-generic, we get a $(\rho +1)$-dimensional family of infinitesimal deformations of $\Lambda_s$ 
induced by all columns. 
Hence, one column corresponds to the deformation in the scroll $X_L$. 
The corresponding section vanishes at the vertex. \qed
\end{proof}

As in the case $V(\chi_s)_1=D_s$, we get the following Torelli-type theorem using Remark \ref{familyAsRulingOfTangentCones}. 

\begin{corollary}
\label{corollaryRecovering}
 A Brill--Noether general canonically embedded curve $C$ is uniquely determined by the family $\family$. 
 More precise, the canonical curve $C$ is a component of the scheme of first- or second-order foci 
 of the family $\family$ induced by the Brill--Noether locus $W_d(C)$ and (the smooth locus of) its singular locus $W^1_d(C)$ of dimension at least one (equivalently $\lceil \frac{g+3}{2} \rceil \le d\le g-1$). 
 %We can recover the canonical curve $C$ from any Brill--Noether loci $W^1_d(C)$. 
 %To be more precise, the family $\family$ determines the family $\underline{\Gamma}$ and 
 %hence, the canonical curve.
\end{corollary}

\section{The first-order focal map}
\label{firstorderFocalMap}

For a general curve $C$ and a sufficiently general point $s\in S$, 
the rank one locus of the focal map $\chi_s$ at $s$ is either $d$ points or a rational normal curve. In the second case, the focal matrix at $s$ is catalecticant (see Corollary \ref{rankOneLocusFocalMap}). 

As mentioned above, the articles \cite{CS95} and \cite{CS00} of Ciliberto and Sernesi 
are the extremal cases ($d=g-1$ and $\rho=1$, respectively), 
where the rank one locus is always a rational normal curve. 
We propose the following question. 

\begin{question}
\label{rank1locusGeneralCurve}
When is the focal matrix $\chi_s$ catalecticant for a general curve $C$ and a sufficiently general point $s\in S$?
\end{question}

We conjecture that only in the extremal cases $d=g-1$ and $\rho=1$ the rank one locus of $\chi_s$ 
is a rational normal curve for a general curve $C$ and a general point $s\in S$. For the rest of this section we explain the reason for our conjecture. 

Let $C\subset \PP^{g-1}$ be a canonically embedded curve of genus $g$ and let $L\in W^1_d(C)$ be a smooth point such that the rank one locus of the focal matrix $\chi_s:T_{S,s}\otimes \OO_{\Lambda_s} \to \Ncal_{\Lambda_s/\PP^{g-1}}$ is a rational normal curve $\Gamma_s$ in $\PP^{d-2}$ for $s\in |L|$ sufficiently general. Let $X_L=\bigcup_{s\in |L|} \overline{D_s}$ be the scroll swept out by the pencil $|L|$. We get a rational surface 
$$\
\Gamma_L=\overline{\bigcup_{s\in |L| \mathrm{\ gen}} \Gamma_s}\subset X_L
$$ 
defined as in the previous section. The rational normal curve $\Gamma_s$ intersects the vertex $V$ of $X_L$ in $\rho=\rho(g,d,1)$ points by Proposition \ref{RNCintersectionVertex}. Note that the scroll $X_L$ is a cone over $\PP^1\times \PP^{h^1(C,L)-1}$ with vertex $V$. Hence, projection from the vertex $V$ yields a rational surface in $\PP^1\times \PP^{h^1(C,L)-1}$ whose general fiber in $\PP^{h^1(C,L)-1}$ is again a rational normal curve. 
We have shown the following proposition. 

\begin{proposition}
\label{curveOnRationalSurfaceSmallDegree}
 Let $C\subset \PP^{g-1}$ be a canonically embedded curve of genus $g$ and let $L\in W^1_d(C)$ be a smooth point such that the rank one locus of the focal matrix $\chi_s$ is a rational normal for $s\in |L|$ sufficiently general. Then, the image of $C$ in $\PP^1\times \PP^{h^1(C,L)-1}$ given by $|L|\times |\omega_C\otimes L^{-1}|$ lies on a rational surface of bidegree $(d', h^1(C,L)-1)$ for some $d'$. 
\end{proposition}

\begin{proof}
 The proposition follows from the preceding discussion. We only note that the map given by $|L|\times |\omega_C\otimes L^{-1}|$ is the same as the projection of $\PP^{g-1}$ along the vertex $V$ of the canonically embedded $C$.  \hspace{5.8cm} \qed
\end{proof}

\begin{example}
 We explain the above circumstance for a curve $C$ of genus $8$ with a line bundle $L\in W^1_6(C)$. The residual line bundle $\omega_C\otimes L^{-1}$ has degree $8$ and $H^0(C,\omega_C\otimes L^{-1})$ is three-dimensional. Let $C'$ be the image of $C$ in $\PP^1\times \PP^2$ given by $|L|\times |\omega_C\otimes L^{-1}|$. We think of $C'\to \PP^1$ as a one-dimensional family of six points in the plane. If our assumption of Proposition \ref{curveOnRationalSurfaceSmallDegree} is true, the six points lie on a conic in every fiber over $\PP^1$. 
 Computing a curve of genus $8$ with a $g^1_6$ in \texttt{Macaulay2} shows that these conics do not exist. Hence, our assumption of Proposition \ref{curveOnRationalSurfaceSmallDegree}, that is, the rank one locus of the focal matrix $\chi_s$ is a rational normal curve for $s\in |L|$ sufficiently general, does not hold for a general curve.  
\end{example}

If $\rho(g,d,1)=2d-g-2\ge 2$ and $d<g-1$, we do not expect the existence of such a rational surface for a curve of genus $g$ and a line bundle of degree $d$ as above. Indeed, $m$ general points in $\PP^{r}$ do not lie on a rational normal curve if $m>r+3$. But the inequality $\rho(g,d,1)=2d-g-2\ge 2$ implies $d>(h^1(C,L)-1)+3$.
Using our \texttt{Macaulay2} package (see \cite{BH}), we could show in several examples ($(g,d)=(8,6),(9,7),(10,8),(9,6)$) that the rational surface of bidegree $(d', h^1(C,L)-1)$ of Proposition \ref{curveOnRationalSurfaceSmallDegree} does not exist. This confirms our conjectural behaviour of the first-order focal map.

\begin{acknowledgement}
I would like to thank Edoardo Sernesi for sharing his knowledge about focal schemes and 
for valuable and enjoyable discussions while my visit at the university Roma Tre. 
The author was partially supported by the DFG-Grant SPP 1489 Schr. 307/5-2.
\end{acknowledgement}

\end{document}